\documentclass{article}

\title{\textbf{On the sum of $k$-th powers in terms of earlier sums}}
\author{Steven J. Miller (Williams College, sjm1@williams.edu) and\\ Enrique Trevi\~no (Lake Forest College, trevino@lakeforest.edu)}

\usepackage{amsmath,amssymb,amsthm,verbatim}

\newcounter{exercise}
\newtheorem{theorem}{Theorem}

\begin{document}
\maketitle

\begin{abstract}
For $k$ a positive integer let $S_k(n) = 1^k + 2^k + \cdots + n^k$, i.e.,
$S_k(n)$ is the sum of the first $k$-th powers. Faulhaber conjectured
(later proved by Jacobi) that for $k$ odd, $S_k(n)$ can be written as a
polynomial of $S_1(n)$, and for $k$ even, $S_k(n)$ can be written as $S_2(n)$ times a polynomial of $S_1(n)$ for example $S_3(n) = S_1(n)^2, S_4(n) = S_2(n)(\frac{6}{5}S_1(n) - \frac{1}{5})$. We give a proof of a variant of this result, namely that for any $k$ there is a polynomial $g_k(x,y)$ such
that $S_k(n) = g(S_1(n), S_2(n))$. The novel proof yields a recursive formula to
evaluate $S_k(n)$ as a polynomial of $n$ that has roughly half the number of terms
as the classical recursive formula that uses Pascal's identity.\end{abstract}



\section{Introduction}

One of the standard problems in elementary number theory is the evaluation of the $k$\textsuperscript{th} power sums: $S_k(n) = \sum_{m=1}^n m^k$. Almost everyone has heard the story of Gauss, at a young age, cleverly finding $S_1(100)$ by noting that if we write the numbers in reverse order and add them $$2 S_1(100) \ = \ (1 + 100) + (2 + 99) + \cdots + (99 + 2) + (100+1) \ = \ 100 \cdot 101, $$ and thus $S_1(100) = 5050$. The goal of this note is to explore consequences of similar tricks to find interesting formulas for other sums.

Our starting point is the well-known observation that $S_k(n)$ is a polynomial with rational coefficients of degree $k+1$. Frequently we use Lagrange interpolation, considering $S_k(0), S_k(1), \dots, S_k(k)$, to find a polynomial of degree $k+1$ that fits the curve, and then prove the polynomial works for all $n$ by induction. One could also prove it by using the Binomial Theorem to see that $$(n+1)^{k+1} - n^{k+1}\ =\ \sum_{\ell=0}^k {k+1\choose \ell}n^{\ell},$$ and then telescoping to get Pascal's identity \cite{Pascal}:
$$(n+1)^{k+1} - 1\ =\ \sum_{\ell = 0}^k {k+1\choose \ell}S_{\ell}(n);$$
thus
\begin{equation}\label{Pascaltrick}
S_k(n)\ =\ \frac{(n+1)^{k+1} - 1 - \sum_{\ell = 0}^{k-1} {k+1\choose \ell}S_{\ell}(n)}{k+1}.
\end{equation}
Note that under this technique, to find $S_k(n)$ as a polynomial of degree $k+1$ we need to have $S_1(n)$, $S_2(n)$, $\dots$, $S_{k-1}(n)$ evaluated.

An interesting variation was considered by Faulhaber, claiming in \cite{Faul} that \emph{for $k$ odd}, there is a polynomial $f_k(x)$ such that $\sum_{m=1}^n m^k = f_k(S_1(n)).$ This result was first proved by Jacobi \cite{Jacobi}, and Knuth gives a list of many Faulhaber polynomials in \cite{Knuth}. We include a few below, using $A(n) = S_1(n)$.
\begin{align*}
S_3(n) &\ = \  A^2(n)\\
S_5(n) &\ = \  \frac{4 A^3(n) - A^2(n)}{3}\\
S_7(n) &\ = \  \frac{12A^4(n) - 8A^3(n) + 2A^2(n)}{6}.
\end{align*}

If we confine ourselves to polynomials in $S_1(n)$, this is the end of the story, but what if we consider $S_2(n)$ as well? Is adding $S_2(n)$ enough to give us a polynomial $g_k(x,y)$ such that $S_k(n) = g_k(S_1(n),S_2(n))$? For $k$ odd, we have Jacobi's result from above, but it's also a classical result that for $k$ even, one can write $S_k(n)$ as the product of $S_2(n)$ with a polynomial (with rational coefficients) in $S_1(n)$ (see \cite{Beardon} or \cite[Theorem 3.5.3, p.~118-119]{pras} for elementary proofs). In this paper, we will describe a different algorithm to find $S_k(n)$ as a polynomial of $S_1(n)$ and $S_2(n)$. In particular, we will show that, for any positive integer $r$, there exist rationals $d_r, d_{r+1}, \ldots, d_{2r-1}, e_{r+1}, e_{r+2}, \ldots, e_{2r}$ such that
\begin{equation*}\label{odd case0}
S_{2r+1}(n) \ = \  \frac{r+1}{2}\left(S^2_r(n) - \sum_{i=r}^{2r-1} d_i S_i(n)\right),
\end{equation*}
and
$$S_{2r+2}(n) \ = \  \frac{(r+1)(r+2)}{2r+3}\left(S_r(n)S_{r+1}(n) - \sum_{i=r+1}^{2r} e_i S_i(n)\right).$$
These equations imply that we can compute $S_k(n)$ as a polynomial of degree $k+1$. One difference with \eqref{Pascaltrick} is that this identity has a weighted sum of $S_i(n)$ for $\lfloor\frac{k-1}{2}\rfloor \le i \le k-2$ as opposed to a weighted sum including all $i \le k-1$.

Furthermore, we get a different proof of the following classical theorem

\begin{theorem}\label{mainTHM}
Let $k$ be a positive integer. There exists a polynomial $g_k \in \mathbb{Q}[x,y]$ such that $g_k(0,0) = 0$  and
$$\sum_{m=1}^n m^k \ = \  g_k(S_1(n), S_2(n)).$$
\end{theorem}

To illustrate some examples, letting $A(n) = S_1(n)$ and $B(n) = S_2(n)$, we find
\begin{align*}
S_3(n) &\ = \  A^2(n)\\
S_4(n) &\ = \  \frac{6}{5} A(n)B(n) -\frac{1}{5}B(n)\\
S_5(n) &\ = \  \frac{3}{2}B^2(n) - \frac{1}{2}A^2(n).\\
S_6(n) &\ = \  \frac{12}{7}A^2(n)B(n) - \frac{6}{7}A(n)B(n) + \frac{1}{7}B(n).
\end{align*}

In the next section we prove our claim, giving the motivation for our choice of $g_k$. There are other possible constructions, which we discuss afterwards, but we were led to this exploration by noting that we can easily find $S_3(n)$ knowing just $S_1(n)$ by cleverly expanding and having sums of squares perfectly cancel. 

\section{Proof of Theorem}

We begin with a few definitions, and then some motivation. For a non-negative integer $k$, define $c_{0k}, c_{1k},\\ \ldots, c_{(k+1)k} \in \mathbb{Q}$ such that
$$S_k(n) \ = \  c_{(k+1)k} n^{k+1} + c_{kk} n^k + \cdots + c_{1k} n + c_{0k}.$$
From \eqref{Pascaltrick} we can easily deduce that $c_{(k+1)k} = \frac{1}{k+1}$, $c_{kk} = \frac{1}{2}$ and $c_{0k} = 0$.\footnote{Bernoulli found an expression for all the coefficients in terms of Bernoulli numbers, but we won't need them in this article.}
Therefore we have
\begin{equation}\label{coeffs}
S_k(n) \ = \  \frac{1}{k+1} n^{k+1} + \frac{1}{2}n^k + c_{(k-1)k} n^{k-1} + \cdots + c_{1k} n.
\end{equation}

\begin{proof}[Proof of Theorem \ref{mainTHM}]
For $k = 1$ we have $g_1(x,y) = x$, since $\sum_{m=1}^n m^1 = A(n)$; similarly $g_2(x,y) = y$.
We show the idea for the proof by calculating $S_3(n)$ and $S_4(n)$. Consider
\begin{align*}
A^2(n) \ = \  \left(\sum_{m=1}^n m\right)^2 &\ = \  \sum_{m=1}^n m^2 + 2\sum_{m=1}^n m \sum_{\ell=1}^{m-1} \ell
\ = \  B(n) + 2\sum_{m=1}^n m \left(\frac{(m-1)m}{2}\right) \\&\ = \  \sum_{m=1}^n m^2 + \sum_{m=1}^n m^3 - \sum_{m=1}^n m^2 \ = \  \sum_{m=1}^n m^3.
\end{align*}
Therefore $S_3(n) = A^2(n)$, so $g_3(x,y) = x^2$. The perfect cancelation of the sum of squares above shows that $S_3(n)$ is just a function of $A(n)$, and was the motivation to see how easily other power sums could be expressed using just $A(n)$ and $B(n)$.

\ \\
For the fourth power sum, consider
\begin{align*}
A(n)B(n) &\ = \  \left(\sum_{m=1}^n m\right)\left(\sum_{m=1}^n m^2\right) \ =\ \sum_{m=1}^n m^3 + \sum_{m=1}^n m\sum_{\ell=1}^{m-1} \ell^2 + \sum_{m=1}^n m^2\sum_{\ell=1}^{m-1} \ell\\
&\ = \  A^2(n) + \sum_{m=1}^n m\frac{(m-1)m(2m-1)}{6} + \sum_{m=1}^n m^2\frac{(m-1)m}{2}\\
&\ = \  A^2(n) + \sum_{m=1}^n\frac{2m^4-3m^3+m^2}{6} + \sum_{m=1}^n \frac{m^4-m^3}{2}\\
&\ = \  A^2(n) + \frac{5}{6}\sum_{m=1}^n m^4 - \sum_{m=1}^n m^3 + \frac{1}{6}\sum_{m=1}^n m^2\\
&\ = \  A^2(n) +\frac{5}{6} \sum_{m=1}^n m^4 - A^2(n) + \frac{1}{6} B(n).
\end{align*}
Therefore
$$\sum_{m=1}^n m^4 \ =\ \frac{6}{5}A(n)B(n) - \frac{1}{5}B(n).$$
Hence $g_4(x,y) = \frac{6}{5}xy - \frac{1}{5}y.$

\ \\
These two cases suggest the general approach: consider an appropriate product of known sums, expand the product in terms of a diagonal and non-diagonal sum, and isolate the new sum in terms of previous ones which we know.

We proceed by induction. Assume that the theorem is true for $k\le 2r$ for some $r \ge 2$, then
\begin{align*}
S^2_{r}(n) \ = \  \left(\sum_{m=1}^{n} m^{r}\right)^2 &\ = \  \sum_{m=1}^n m^{2r} + 2\sum_{m=1}^n m^r \sum_{\ell =1}^{m-1} \ell^r\\
&\ = \  \sum_{m=1}^n m^{2r} + 2\sum_{m=1}^n \sum_{i=1}^{r+1} c_{i r} m^r (m-1)^i\\
&\ = \  \sum_{i=r}^{2r+1} d_{i}\sum_{m=1}^n m^i = \sum_{i=r}^{2r+1} d_i S_i(n),
\end{align*}
for some rational coefficients $d_i$. From \eqref{coeffs} we can deduce that $d_{2r+1} = \frac{2}{r+1}$, and $$d_{2r} \ = \  1 + 2(c_{rr} - (r+1)c_{(r+1)r}) \ = \  1+2\left(\frac{1}{2}-\frac{r+1}{r+1}\right) \ = \  0.$$ Thus
\begin{equation}\label{odd case}
S_{2r+1}(n) \ = \  \frac{r+1}{2}\left(S^2_r(n) - \sum_{i=r}^{2r-1} d_i S_i(n)\right).
\end{equation}
Therefore there exists $g_{2r+1}(x,y) \in \mathbb{Q}[x,y]$ such that $g_{2r+1}(A(n),B(n)) = S_{2r+1}(n)$. We can also verify $g_{2r+1}(0,0)=0$ from \eqref{odd case}.

The above deals with the odd cases. Now let's deal with the even cases.
\begin{align*}
S_{r}(n)S_{r+1}(n) &\ = \  \left(\sum_{m=1}^{n} m^{r}\right)\left(\sum_{m=1}^n m^{r+1}\right) \ = \  \sum_{m=1}^n m^{2r+1} + \sum_{m=1}^n m^r \sum_{\ell =1}^{m-1} \ell^{r+1} + \sum_{m=1}^n m^{r+1} \sum_{\ell =1}^{m-1} \ell^{r}\\
&\ = \  \sum_{m=1}^n m^{2r+1} + \sum_{m=1}^n \sum_{i=1}^{r+2} c_{i(r+1)} m^r (m-1)^i + \sum_{m=1}^n \sum_{i=1}^{r+1} c_{ir} m^{r+1} (m-1)^i\\
&\ = \ \sum_{i=r+1}^{2r+2} e_i\sum_{m=1}^n m^i \ = \  \sum_{i=r+1}^{2r+2}e_i S_i(n),
\end{align*}
for some rationals $e_i$. From \eqref{coeffs} we can deduce that $e_{2r+2} = \frac{1}{r+1}+\frac{1}{r+2} = \frac{2r+3}{(r+1)(r+2)}$, and $$e_{2r+1} \ = \  1-(r+2)c_{(r+2)(r+1)} +c_{(r+1)(r+1)}+c_{r r} - (r+1)c_{(r+1)r} \ = \  1-\frac{r+2}{r+2} + \frac{1}{2} + \frac{1}{2} - \frac{r+1}{r+1} \ = \  0.$$ Therefore
$$S_{2r+2}(n) \ = \  \frac{(r+1)(r+2)}{2r+3}\left(S_r(n)S_{r+1}(n) - \sum_{i=r+1}^{2r} e_i S_i(n)\right).$$
Since $S_i(n) = g_i(A(n),B(n))$, for $r\le i \le 2r$, $S_{2r+2}(n)$ can also be expressed as a polynomial of $A(n), B(n)$ with rational coefficients.
\end{proof}

\section{Future Work}

We end with a brief discussion of some additional questions that one can investigate. Now that we know the existence of a polynomial, it is natural to ask if it is unique, and if not then how many polynomials there are. It cannot be unique for all $k$; easy candidates occur when $k$ is odd, as Jacobi proved Faulhaber's conjecture that for $k$ odd the corresponding power sum can be written as a polynomial in just $A(n)$. For example, $S_5(n) = \frac{3}{2}B^2(n) - \frac{1}{2}A^2(n)$ shows that we can sometimes have multiple representations.

We sketch another approach, which might generate many different solutions. We know $S_k(n)$ is a polynomial in $n$ of degree $k+1$. As $A(n)$ is a polynomial of degree 2 and $B(n)$ of degree 3, we have $A^a(n) B^b(n)$ is a polynomial in $n$ of degree $2a+3b$. This suggests, for $2a + 3b = k+1$, looking at $$S_k(n) - \alpha_{a,b} A^a(n) B^b(n).$$ By choosing $\alpha_{a,b}$ appropriately what remains is a polynomial of degree at most $k$, and hence is of lower degree than $S_k(n)$. We would like to say we can continue by induction, replacing degree by degree the terms that remain with products of $A(n)$ and $B(n)$, but there is an obstruction. We can do this trick to remove the leading term only if its degree is at least 2; thus we could reach a situation where after several steps we are left with a linear term in $n$, and that \emph{cannot} be written as a polynomial in $A(n)$ and $B(n)$; this method would thus work only if a miracle occurred and there was no linear term. We leave the interesting question of how many polynomials work to the reader, as well as whether or not any of these polynomials have coefficients with interesting combinatorial significance.

An interesting question is whether or not we can write all power sums from some point onward as a polynomial of a finite number of fixed power sums. For example, if $k$ is sufficiently large is $S_k(n)$ a polynomial in $S_1(n), S_4(n)$ and $S_6(n)$? Or of $S_2(n), S_3(n)$ and $S_4(n)$? A related problem was solved by Beardon \cite{Beardon} who showed that given two distinct positive integers $i,j$, there exists a polynomial $T(x,y)$ with integer coefficients such that $T(S_i(n), S_j(n)) = 0$.

\section*{Acknowledgements}
We'd like to thank Bernd Keller and some anonymous referees for pointing out relevant references and helful suggestions that improved the paper.





\end{document}